\documentclass{aims} 
\usepackage{amsmath}
\usepackage{amsthm}
\usepackage{paralist}
\usepackage{epsfig} 
\usepackage{graphicx}  
\usepackage{epstopdf} 
\usepackage[colorlinks=true]{hyperref}
\usepackage{url}
\usepackage[pagewise]{lineno}

\hypersetup{urlcolor=blue, citecolor=red}
\allowdisplaybreaks

\def\currentvolume{X}
 \def\currentissue{X}
  \def\currentyear{200X}
   \def\currentmonth{XX}
    \def\ppages{X--XX}

\newtheorem{theorem}{Theorem}[section]

\newtheorem{proposition}[theorem]{Proposition}

\theoremstyle{definition}
\newtheorem{definition}[theorem]{Definition}

\newcommand{\bx}{\mathbf{x}}

\renewcommand{\div}{\ensuremath{\mathrm{div}}}

\title[Determining  measures and integrals of birational maps]{Detecting and determining preserved measures and integrals of birational maps}

\author[Celledoni, Evripidou, McLaren, Owren, Quispel and Tapley]{}


 \email{Elena.Celledoni@ntnu.no}
 \email{pamb0sd16@gmail.com}
 \email{D.McLaren@latrobe.edu.au}
 \email{Brynjulf.Owren@ntnu.no}
 \email{R.Quispel@latrobe.edu.au}
 \email{Benjamin.Kwanen.Tapley@dnb.no}

\thanks{$^*$Corresponding author: D. I. McLaren}

\begin{document}

\centerline{\scshape Elena Celledoni}
\medskip
{\footnotesize
 \centerline{Department of Mathematical Sciences, NTNU}
   \centerline{7491 Trondheim, Norway}
} 

\medskip

\centerline{\scshape Charalambos Evripidou
}
\medskip
{\footnotesize
 \centerline{Department of Mathematics and Statistics, University of Cyprus}
   \centerline{1678 Nicosia, Cyprus}
}

\medskip

\centerline{\scshape David I. McLaren$^{*}$}
\medskip
{\footnotesize
 \centerline{Department of Mathematical and Physical Sciences, LaTrobe University}
   \centerline{Bundoora, VIC 3083, Australia}
}

\medskip

\centerline{\scshape Brynjulf Owren}
\medskip
{\footnotesize
 \centerline{Department of Mathematical Sciences, NTNU}
   \centerline{7491 Trondheim, Norway}
}

\medskip

\centerline{\scshape G. R. W. Quispel}
\medskip
{\footnotesize
  \centerline{Department of Mathematical and Physical Sciences, LaTrobe University}
   \centerline{Bundoora, VIC 3083, Australia}
}

\medskip
\centerline{\scshape Benjamin K. Tapley}
\medskip
{\footnotesize
  \centerline{Department of Mathematical Sciences, NTNU}
   \centerline{7491 Trondheim, Norway}
}

\bigskip

 \centerline{(Communicated by the associate editor name)}


\begin{abstract}
In this paper we use the method of discrete Darboux polynomials to calculate preserved measures and integrals of rational maps. The approach is based on the use of cofactors and Darboux polynomials and relies on the use of symbolic algebra tools. Given sufficient computing power, most, if not all, rational preserved integrals can be found (and even some non-rational ones).
 We show, in a number of examples, how it is possible to use this method to both determine and detect preserved measures and integrals of the considered rational  maps, thus lending weight to a previous ansatz \cite{CEMOQTV}. Many of the examples arise from the Kahan-Hirota-Kimura discretization of completely integrable systems of ordinary differential equations.

\end{abstract}
\maketitle

\def\eqalign#1{\null\,\vcenter{\openup\jot \mathsurround=0pt \ialign{\strut
     \hfil$\displaystyle{##}$&$ \displaystyle{{}##}$\hfil \crcr#1\crcr}}\,}

\def\R{\mathbb{R}}

\section{Introduction}
“Most of science is a search for simple, stable properties that can answer questions which interest us.” (Quote from F. Wilczek’s book \cite{Wil}).
The above quote certainly applies to the study of ordinary differential equations (ODEs). Starting with Galileo, Kepler and Newton, and later luminaries such as Hamilton and Lagrange, the search for preserved first integrals  and preserved volume forms has a long and very distinguished history.
For the area of ordinary difference equations and mappings, by comparison, the corresponding search for preserved integrals and preserved volume forms is arguably still in its infancy and is mainly carried out within two subfields:  discrete integrable systems \cite{FV},  and numerical integration of ODEs, in particular geometric numerical integration \cite{HLW}. 
In the discrete integrable systems subfield, two of the earliest examples of the discovery of discrete maps preserving integrals are the 2D area-preserving McMillan map, preserving a polynomial integral \cite{McM}, subsequently generalized to the 2D measure-preserving QRT map, preserving a rational integral \cite{QRT1,QRT2}.
In the area of numerical discretization of ODEs, one of the earliest methods used is the birational Kahan-Hirota-Kimura discretization of first order quadratic ODEs \cite{celledoni18gai,HK,KH,K}, subsequently followed by other birational discretizations of ODEs of higher degree and/or higher order, such as polarization methods \cite{CMMOQpol} and the methods of Hone and Quispel \cite{HQ}.
In a recent Letter \cite{CEMOQTV}, we developed a theory for detecting and calculating integrals and preserved measures of rational maps, and presented three examples illustrating its use. 
The algorithm we use  is based on the notion of discrete Darboux polynomials and an accompanying cofactor, and it essentially only requires the solution of linear systems of equations.
The details of the algorithm are presented in the next section, culminating in Section 2.5.
In this paper we shall build further on the work of  Celledoni et al. \cite{CEMOQTV}. After introducing notation and preliminary results, such as the Kahan map, in section~\ref{preliminaries}, we shall
 present a total of 10 examples. In section 3, Examples 1 to 6 involve the discretisation of an ODE using Kahan's method or the Hone-Quispel method, whereas Example 7 comes from the area of discrete integrable systems. For all these discrete systems our method is used to determine a sufficient number of Darboux polynomials of the system. A specific one of these Darboux polynomials yields a preserved measure, whereas the set of all Darboux polynomials is used as building blocks to construct $k-1$ first integrals (where $k$ is the dimension of the discrete system). Most of these integrals turn out to be rational, but Example 5 exemplifies the construction of a non-rational integral using Darboux polynomials. In section 4, three additional examples are given where our method is extended to find specific parameters of a discrete system for which it has additional Darboux polynomials. These latter are also computed, as well as preserved measures and first integrals of the discrete systems.  Finally, we prove that for any quadratic Hamiltonian ODE, the modified Hamiltonian as well as the modified preserved measure of the Kahan map are both found using a particular ansatz for the choice of cofactor.

\section{Preliminaries}\label{preliminaries}

  In this section, we provide some preliminaries on Kahan's discretization, on measure preservation and superintegrability, on Darboux polynomials for ODEs and for discrete birational maps, and on the discrete cofactor ansatz that we will use in the remainder of the paper:

\subsection{Kahan's discretization}
Kahan \cite{K} proposed a numerical method designed for quadratic systems of ordinary differential
equations in $\mathbb{R}^n$ written in component form as
\begin{equation} \label{quadvectorfield}
\frac{d x_i}{d t} = \sum_{j,k} a_{ijk} x_j x_k + \sum_j b_{ij} x_j + c_i,\quad i=1,\ldots,n,
\end{equation}
where $a_{ijk}, b_{ij}, c_i$ are arbitrary constants and all summation indices are ranging from $1$ to $n$.
The method of Kahan, also known as the Hirota--Kimura discretization \cite{HK,KH}, is a one-step method
$(x_1,\ldots,x_n)\mapsto (x_1',\ldots,x_n')$ where
\begin{equation} \label{Kahansmethod}
   \frac{x_i'-x_i}{h} = \sum_{j,k} a_{ijk} \frac{x_j'x_k+x_jx_k'}{2} + \sum_j  b_{ij}\frac{x_j+x_j'}{2} + c_i,\quad i=1,\ldots,n,
\end{equation}
where $h$ denotes the discrete time step.
The method (\ref{Kahansmethod}) is linearly implicit and so is its inverse, hence it defines a birational map $\phi_h$.

  Much of the recent interest in Kahan's method stems from its ability to preserve modified first integrals and measures
of the underlying quadratic differential equation \cite{PPS,CMOQ, CMMOQ}. But even in cases where there are strong indications
that Kahan's method preserves such a nearby invariant, it is not necessarily an easy task to determine its closed form.

\subsection{Measure preservation and superintegrability}
Consider the ODE
\begin{equation}\label{ode_p2}
\frac{d\bx}{dt} = f(\bx), \quad \bx:=(x_1,x_2,\dots,x_n) \in \mathbb{R}^n.
\end{equation}

  \begin{proposition}  (Liouville, \cite{ARNOLD})
The smooth function $M: \mathbb{R}^n \rightarrow \mathbb{R}$ is the density of an invariant of the ODE (\ref{ode_p2}) iff 
$$\
\div(Mf)=0.
$$ 
\end{proposition}

\begin{definition} \cite{VKQTV}
  A vector field on $\mathbb{R}^n$ is superintegrable if it admits $n-1$ functionally independent constants of motion.
\end{definition}
We now consider the birational map $$\bx' = \phi(\bx),$$ with $\bx = (x_1,x_2,\dots,x_n) \mbox{ and } \bx' = (x'_1,x'_2,\dots,x'_n)$ elements of $\mathbb{R}^n$.
\begin{definition}
The map $\phi$ is measure preserving if there exists a smooth function $\Pi$ such that $$\Pi(\bx')J(\bx) =  \Pi(\bx),$$ where $J$ is the Jacobian determinant of $\phi$ $$J = \left| \frac{\partial \phi_i}{\partial x_j} \right|.$$ The preserved measure is given by $\Pi(\bx)\,d\bx$ where $d\bx= dx_1 \wedge \dots \wedge dx_n$.
\end{definition}
As mentioned in \S3 of \cite{VKQTV}, for a discrete map in $\mathbb{R}^n$, the existence of $n-1$ integrals is not enough to claim (super) integrability:
\begin{definition} \cite{VKQTV}
 An $n$-dimensional map is superintegrable if it has $n-1$ constants of motion and it is measure-preserving.
\end{definition}

\subsection{Darboux polynomials for ordinary differential equations}
For references on this topic, the reader is referred to Cheze \& Combot \cite{CandC} and references therein.

  Consider an ODE $\frac{d\bx}{dt}=f(\bx)$. The polynomial $\tilde{P}(\bx)$ is defined to be a Darboux polynomial of the ODE if there exists a polynomial function $\tilde{C}(\bx)$ s.t. $$\frac{d}{dt}\tilde{P}(\bx)= \tilde{C}(\bx) \tilde{P}(\bx).$$ Here $\frac{d}{dt}\tilde{P}(\bx)= f(\bx).\nabla \tilde{P}(\bx) $, and $\tilde{C}$ is called the cofactor of $\tilde{P}$.
Note that if
\begin{equation*}
\frac{d}{dt}\tilde{P}_i(\bx) = \tilde{C}_i(\bx) \tilde{P}_i(\bx), \qquad i=1,\dots,l, 
\end{equation*}
then
\begin{equation*}
\frac{d}{dt}\left(\prod_{i}^{} \tilde{P}_i^{\alpha_i}(\bx) \right) = \left( \sum_{i}^{} \alpha_i\tilde{C}_i(\bx) \right) \left(\prod_{i}^{} \tilde{P}_i^{\alpha_i}(\bx) \right), 
\end{equation*}
so that if 
$$
\sum_{i}^{} \alpha_i\tilde{C}_i(\bx) = 0,
$$ 
then 
$$
\prod_{i}^{} \tilde{P}_i^{\alpha_i}(\bx),$$
 is an integral of the vector field $f$.
We list some examples of cofactors and the geometric interpretation of the corresponding Darboux polynomials:
  \begin{enumerate}
  \item $\tilde{C}(\mathbf{\bx}) \equiv 0 \Rightarrow \dot{\tilde{P}}=0 \Rightarrow \tilde{P}$ is a first integral.
  \item $\tilde{C}(\mathbf{x}) \equiv \tilde{C} \Rightarrow \frac{d}{dt}{\tilde{P}}=\tilde{C}\tilde{P} \Rightarrow \tilde{P}(\mathbf{x}(t)) = \tilde{P}(\mathbf{x}(0)) e^{\tilde{C}t} \Rightarrow \tilde{P}$ defines a foliation.
  \item$\tilde{C}(\mathbf{x}) \equiv \div f(\mathbf{x})  \Rightarrow \tilde{P}$ defines the preserved measure $\frac{d\mathbf{x}}{\tilde{P}(\mathbf{x})}$ of the vector field $f$.
  \item $\tilde{C}(\mathbf{\bx}) \equiv \tilde{C}(\mathbf{\bx})$, i.e. the general case $\Rightarrow \tilde{P}(\bx)=0$ defines a second integral, i.e. an algebraic invariant hypersurface.
  \end{enumerate}
  
\subsection{Darboux polynomials for discrete rational maps}
Consider a rational map $\bx'=\phi(\bx)$. We define the polynomial $P(\bx)$ to be a (discrete) Darboux polynomial of the map $\phi$ if there exists a rational function $C(\bx)$ s.t. $P(\bx)$ satisfies the cofactor equation: 
\begin{equation}\label{cofeqn}
P(\bx') = C(\bx) P(\bx), 
\end{equation} 
where the form of $C(\bx)$ will be prescribed below.

  Note that if $$P_i(\bx') =  C_i(\bx) P_i(\bx), \quad i=1,\dots,k,$$ then 
$$
\left( \prod_{i}^{} P_i^{a_i}(\bx') \right) = \left( \prod_{i}^{} C_i^{a_i}(\bx) \right) \left( \prod_{i}^{} P_i^{a_i}(\bx)  \right), 
$$
so that if 
$$
	 \prod_{i}^{} C_i^{a_i}(\bx) \equiv 1,
$$ 
then 
$$
	 \prod_{i}^{} P_i^{a_i}(\bx),
$$ is an integral of the map $\phi$. We will see that this has the consequence that many foliations arising in this way are in factorised form.

We list some examples of discrete cofactors and the geometric interpretation of the corresponding  Darboux polynomials:
\begin{enumerate}
\item $C(\bx) \equiv 1 \rightarrow P'=P \rightarrow P$ is a first integral.
\item $C(\bx) \equiv C \rightarrow P'=CP \rightarrow P(\bx_n)=P(\bx_0)C^n \rightarrow P$ defines a foliation.
\item $C(\bx) \equiv J(\bx) \rightarrow P$ defines the preserved measure $\frac{d\bx}{P(\bx)}$ of the map $\phi$.
\item $C(\bx) \equiv C(\bx)$ i.e. the general case $\rightarrow P(\bx)=0$ defines a second integral, i.e. an algebraic invariant hypersurface. 
\end{enumerate}

  We remark that if the map $\phi$ contains a timestep $h$, such that $\displaystyle{\lim_{h \rightarrow 0} \frac{\phi - 1}{h} = f(\bx)}$, then in the continuum limit 
\begin{equation}\label{contlim}
\begin{aligned}
\tilde{P}(\bx) &= \lim_{h \rightarrow 0} P(\bx), \\
\tilde{C}(\bx) &= \lim_{h \rightarrow 0} \frac{C(\bx)-1}{h}, \\
\div(f(\bx)) &=  \lim_{h \rightarrow 0} \frac{J(\bx)-1}{h},  \\
\tilde{C} &= \lim_{h \rightarrow 0} \frac{C-1}{h}, 
\end{aligned}
\end{equation}
the (discrete) mapping case reduces to the (continuous) ODE case.

\subsection{Jacobian factor ansatz and the algorithm}
The remaining question is how to prescribe the form of the discrete cofactor $C_i(\bx)$. In \cite{CEMOQTV} we introduced the following ansatz: Given a rational map $\phi$ with Jacobian determinant 
$$
J(\bx) = \frac{\prod_{i=1}^{l}K_i^{b_i}(\bx)}{\prod_{j=1}^{k}D_j^{c_j}(\bx)},
$$
where the $K_i$ and $D_j$ are distinct factors, we try all cofactors (up to a certain polynomial degree $d_1$ for the numerator and $d_2$ for the denominator) of the form 
$$
C(\bx) = \pm \frac{\prod_{i=1}^{l}K_i^{f_i}(\bx)}{\prod_{j=1}^{k}D_j^{g_j}(\bx)},
$$
where $f_i, g_j \in \mathbb{N}_0$. This ensures that we need only check a finite number of cofactors. (Note that in maps arising from Kahan's discretization, it follows from (\ref{contlim}) that both $J$ and $C$ are $1+O(h)$).

  For each such cofactor we try all possible Darboux polynomials, (again up to a certain degree). 

  Given the cofactor $C$, the question of whether $P$ exists, and, if so, what it is, only requires solving the linear cofactor equation (\ref{cofeqn}).

\section{Determining preserved measures and integrals} \label{determine}
In this section we apply the aforementioned algorithm to determined preserved measures and integrals for 7 examples.
The algorithm has been implemented in Maple\footnote{Version of 2019 was used  and the codes are adapted to run on computing servers of up to 32 cores with up to 768 GBs of memory.}

\subsection{Example 1: finding measures and integrals of a specific 2D vector field}
This subsection exemplifies the preservation of a modified quadratic Darboux polynomial, and restates a result of \cite{CMOQ} in terms of Darboux polynomials. We study the following two-dimensional vector field
\begin{equation}\label{2Dvf}
\frac{d}{d t}\left( \begin{array}{c}
 x_1 \\ 
 x_2  
 \end{array} \right) 
 =
\left( \begin{array}{c} 
2x_1x_2 - 4x_2\\
-3x_1^2 - x_2^2 +4x_1 +1
  \end{array} \right) 
  .
\end{equation}
The Kahan discretization $\phi_h$ of this vector field is given by
\begin{equation} \label{2DKahan}
\begin{array}{lcl}
\displaystyle{x_1'} &=& \displaystyle{\frac{x_1 + h(2x_1x_2-4x_2) + h^2(2x_1^2-2x_2^2-3x_1-2)}{D({\bf x})}}, \\
\displaystyle{x_2'}  &=& \displaystyle{\frac{x_2 + h(-3x_1^2-x_2^2+4x_1+1) + h^2(4x_1x_2-5x_2)}{D({\bf x})}},
\end{array}
\end{equation}
where the common denominator $D({\bf x})$ is given by
$$D({\bf x}) = 1 + h^2(3x_1^2-x_2^2-8x_1+4).$$
The Jacobian determinant $J$ of the Kahan map (\ref{2DKahan}) is given by 
$$ J = C_1({\bf x}) C_2({\bf x}), $$
where
\begin{align*}
C_1({\bf x}) &= \frac{1+2hx_2+h^2(5-4x_1)}{D({\bf x})}, \\
C_2({\bf x}) &= \frac{1+ C_2^1 h +C_2^2 h^2 + C_2^3 h^3 + C_2^4 h^4}{D^2({\bf x})}.
\end{align*}
where
\begin{align*}
C_2^1 &= -2hx_2, \\
C_2^2 &= 7-20x_1+9x_1^2+x_2^2,  \\
C_2^3 &= 26x_2-16x_1x_2, \\
C_2^4 &= 28-28x_1+7x_1^2+3x_2^2.
\end{align*}
Defining $C_3:=J$, we have used cofactors $C_1, C_2$ and $C_3$ to find the corresponding Darboux polynomials for the Kahan map (\ref{2DKahan})
\begin{align*}
p_{1,1} &= x_1-2, \\
p_{2,1} &= 1-x_1^2-x_2^2 + h^2(\frac{13}{3} - \frac{16}{3}x_1 + x_2^2 + \frac{7}{3}x_1^2 ), \\
p_{3,1} &= 1 + h^2(3x_1^2-x_2^2-8x_1+4),  \\
p_{3,2} &= (x_1-2)\,(1-x_1^2-x_2^2 + h^2(\frac{13}{3} - \frac{16}{3}x_1 + x_2^2 + \frac{7}{3}x_1^2 )).
\end{align*}
Here and below, $p_{i,j}$ denotes the $jth$ Darboux polynomial corresponding to the cofactor $C_i$, i.e., $p_{i,j}$ satisfies 
\begin{equation*}
	p_{i,j}(\mathbf{x}') = C_i(\mathbf{x}) p_{i,j}(\mathbf{x}).
\end{equation*}

\begin{figure}[ht!]
\begin{center}
\includegraphics[width=8cm]{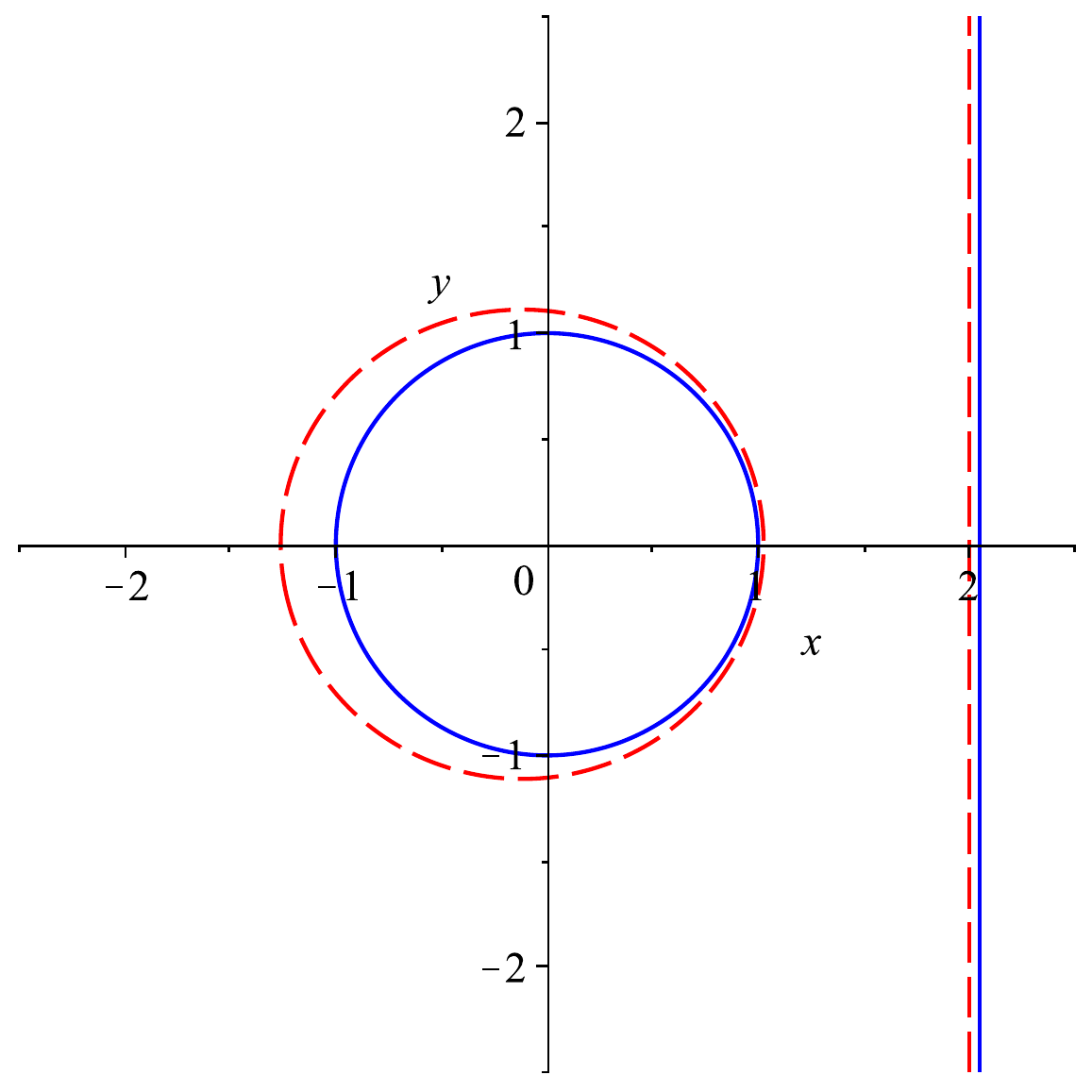}
\end{center}
\caption{Plots of the level sets $p_{1,1}=0$ and $p_{2,1}=0$ of the Kahan map in Example 1 (for $h=\frac{1}{5}$), dotted red. Also shown  are the corresponding second integrals of the ODE, $x_1-2$ and $1-x_1^2-x_2^2$, solid blue.}
\end{figure}

 Note also that here $p_{3,1}({\bf x}) \equiv D({\bf x})$. Hence it turns out that the Kahan map (\ref{2DKahan}) possesses the  Darboux polynomials $p_{1,1}({\bf x})$, $p_{2,1}({\bf x})$ and $p_{3,1}({\bf x})$, and also preserves the measure 
$$
\frac{d\bx }{1 + h^2(3x_1^2-x_2^2-8x_1+4)}. 
$$
 Finally, the Kahan map preserves the first integral
$$\frac{(x_1-2)\,(1-x_1^2-x_2^2 + h^2(\frac{13}{3} - \frac{16}{3}x_1 + x_2^2 + \frac{7}{3}x_1^2 ))}{1 + h^2(3x_1^2-x_2^2-8x_1+4)}.$$

  Taking the continuum limit $h \rightarrow 0$, we now see in hindsight that the vector field (\ref{2Dvf}) possesses two second integrals, i.e. $x_1-2$ and $1-x_1^2-x_2^2$\footnote{This is why the discrete Darboux polynomial $p_{2,1} = 1-x_1^2-x_2^2 + h^2(\frac{13}{3} - \frac{16}{3}x_1 + x_2^2 + \frac{7}{3}x_1^2 )$ is called a {\it modified} Darboux polynomial}. It also preserves the measure $dx_1\wedge dx_2$ and the first integral $H = (x_1-2 )(1-x_1^2-x_2^2)$. The fact that the affine Darboux polynomial $x_1-2$ is preserved by the Kahan map (\ref{2DKahan}) is an example of Theorem 1 of \cite{CEMOQTV}, which states that the Kahan discretization preserves all affine Darboux polynomials in any dimension. On the other hand, the vector field (\ref{2Dvf}) preserves the integral $H$ and the measure $d x_1\wedge d x_2$ implying that (\ref{2Dvf}) is a Hamiltonian vector field with cubic Hamiltonian $H$. Therefore, the fact that the Kahan method preserves a modified Hamiltonian $\tilde{H}$, and the modified densities $p_{3,1}$ and $p_{3,2}$ is a special case of the following theorem\footnote{Additional geometric properties of the Kahan map (\ref{2DKahan}) will be given in a forthcoming preprint by Gubbiotti, Quispel and McLaren.}:
\begin{theorem}
Let $H$ be cubic in $\mathbb{R}^n$, let $K$ be a constant rank $2l$ antisymmetric $n\times n$ matrix and let the vector field be given by $f=K\nabla H(\mathbf{x})$. Then:
\begin{itemize}
	\item[(i)] $\phi_h({\bf x})$ possesses the following two Darboux polynomials, both with cofactor\footnote{We remind the reader that $J$ is the Jacobian determinant of the Kahan map $\phi_h$.} $C_1({\bf x}) = J$:
	\begin{align*}
		p_{1,1} &= H({\bf x})\, \mathrm{det}\,(A({\bf x})) + \frac{1}{3}h \nabla H({\bf x})^t \mathrm{adj}(A({\bf x})) f({\bf x}),  \\
		p_{1,2} &= \mathrm{det}\,(A({\bf x})).
	\end{align*}
	Here $A({\bf x}) = \mathbb{I} - \frac{1}{2}hf'({\bf x})$, and $\mathrm{adj}(A)$ denotes the adjugate of $A$.
	\item[(ii)]moreover, the degree of $p_{1,2}$ is at most $2l$ and the degree of $p_{1,1}$ is at most $2l+3$. If $n=2l$ the degree of $p_{1,1}$ is at most $2l+1$. 
\end{itemize} 

\end{theorem}

\begin{proof}
In the proof of Proposition 4 in \cite{CMOQ}, it is shown that $\phi_h$ possesses the modified integral $\tilde{H} = \frac{p_{1,1}}{p_{1,2}}$. Proposition 5 in \cite{CMOQ} is equivalent to the statement that $p_{1,2}$ is a Darboux polynomial with cofactor $J$. Combining these two results, it follows that $p_{1,1}$ is also a Darboux polynomial with cofactor $J$. Part \textit{(ii)} follows from proposition 4\textit{(i)} in \cite{CMOQ}.

\end{proof}

\subsection{Example 2: An inhomogeneous Nambu system}

We consider the following inhomogeneous Nambu system belonging to the class of systems considered in \cite{celledoni18gai}
\begin{equation}
\frac{d}{dt}\left( \begin{array}{c} x_1 \\ x_2 \\ x_3 \end{array} \right) =
\left( \begin{array}{c} 20\,x_{{1}}x_{{2}}+8\,x_{{1}}x_{{3}}+30\,{x_{{2}}}^{2}+12\,x_{{2}}x_{{
3}}+32\,x_{{1}}+48\,x_{{2}}
 \\-10\,x_{{1}}x_{{2}}-4\,x_{{1}}x_{{3}}-20\,{x_{{2}}}^{2}-8\,x_{{2}}x_{{
3}}-16\,x_{{1}}-32\,x_{{2}}
 \\16\,x_{{1}}x_{{2}}+10\,x_{{1}}x_{{3}}+32\,{x_{{2}}}^{2}+20\,x_{{2}}x_{
{3}}+10\,x_{{1}}+20\,x_{{2}}
 \end{array} \right) 
 ,
\end{equation}
which has the two integrals
\begin{align*}
H_1 &:= x_1^2 + 4x_1x_2 + 3x_2^2,\\
H_2 &:= 4x_2^2 + 5x_2x_3 + x_3^2 + 5x_2 + 8x_3,
\end{align*}
and the preserved measure
$$
\int dx_1\wedge dx_2\wedge  dx_3.
$$
 We consider the Kahan discretization of these equations.
The corresponding Jacobian determinant has the irreducible factorization
\begin{equation}
J= \frac{K_1({\bf x})K_2({\bf x})K_3({\bf x})K_4({\bf x})}{D({\bf x})^4}, 
\end{equation}
where
\begin{align*}
K_1 &= 1+ \left( 10\,x_{{2}}+4\,x_{{3}}+16 \right) h+ \left( -9\,{x_{{1}}}^{2
}-54\,x_{{1}}x_{{2}}-56\,{x_{{2}}}^{2}+20\,x_{{2}}x_{{3}} \right. \\
 & \left. +4\,{x_{{3}}}
^{2}-60\,x_{{1}}-70\,x_{{2}}+32\,x_{{3}}+64 \right) {h}^{2}+ \left( -
150\,x_{{1}}x_{{2}}-60\,x_{{1}}x_{{3}} \right. \\
& \left. -300\,{x_{{2}}}^{2}-120\,x_{{2}}
x_{{3}}-240\,x_{{1}}-480\,x_{{2}} \right) {h}^{3},
 \\
K_2 &= 1+ \left( -6\,x_{{1}}-12\,x_{{2}} \right) h+ \left( 9\,{x_{{1}}}^{2}+
36\,x_{{1}}x_{{2}}+32\,{x_{{2}}}^{2}-8\,x_{{2}}x_{{3}}-4\,{x_{{3}}}^{2} \right. \\
 & \left.-62\,x_{{2}}-32\,x_{{3}}-64 \right) {h}^{2}+ \left( 150\,x_{{1}}x_{{2
}}+60\,x_{{1}}x_{{3}}+300\,{x_{{2}}}^{2}+120\,x_{{2}}x_{{3}} \right. \\
 & \left. +240\,x_{{
1}}+480\,x_{{2}} \right) {h}^{3},
 \\
K_3 &= 1+ \left( 6\,x_{{1}}+12\,x_{{2}} \right) h+ \left( 9\,{x_{{1}}}^{2}+36
\,x_{{1}}x_{{2}}-28\,{x_{{2}}}^{2}-32\,x_{{2}}x_{{3}}-4\,{x_{{3}}}^{2} \right. \\
 & \left. 
-158\,x_{{2}}-32\,x_{{3}}-64 \right) {h}^{2}+ \left( 150\,x_{{1}}x_{{2
}}+60\,x_{{1}}x_{{3}}+300\,{x_{{2}}}^{2}+120\,x_{{2}}x_{{3}} \right. \\
 & \left. +240\,x_{{
1}}+480\,x_{{2}} \right) {h}^{3},
 \\
K_4 &= 1+h \left( -10\,x_{{2}}-4\,x_{{3}}-16 \right) + \left( -9\,{x_{{1}}}^{
2}-18\,x_{{1}}x_{{2}}+16\,{x_{{2}}}^{2}+20\,x_{{2}}x_{{3}} \right. \\
 & \left. +4\,{x_{{3}}
}^{2}+60\,x_{{1}}+170\,x_{{2}}+32\,x_{{3}}+64 \right) {h}^{2}+ \left( 
-150\,x_{{1}}x_{{2}}-60\,x_{{1}}x_{{3}} \right. \\
&  \left. -300\,{x_{{2}}}^{2}-120\,x_{{2}
}x_{{3}}-240\,x_{{1}}-480\,x_{{2}} \right) {h}^{3},
 \\
D &= 1+ \left( -9\,{x_{{1}}}^{2}-36\,x_{{1}}x_{{2}}-70\,{x_{{2}}}^{2}-20\,x
_{{2}}x_{{3}}-4\,{x_{{3}}}^{2}-110\,x_{{2}}-32\,x_{{3}}-64 \right) {h}
^{2} \\
&  + \left( 90\,x_{{1}}{x_{{2}}}^{2}+36\,x_{{1}}x_{{2}}x_{{3}}+180\,{
x_{{2}}}^{3}+72\,{x_{{2}}}^{2}x_{{3}}+294\,x_{{1}}x_{{2}}+60\,x_{{1}}x
_{{3}} \right. \\
 & \left. +588\,{x_{{2}}}^{2}+120\,x_{{2}}x_{{3}}+240\,x_{{1}}+480\,x_{{2}
} \right) {h}^{3}.
\end{align*}
Defining cofactors $C_i = K_i/D, i=1\dots4$, we find the following Darboux polynomials $p_{i,1}$:
\begin{align*}
p_{1,1} &= x_{{1}}+x_{{2}}, \\
p_{2,1} &= 4\,x_{{2}}+x_{{3}}+9, \\
p_{3,1} &= x_{{2}}+x_{{3}}-1, \\
p_{4,1} &= x_{{1}}+3\,x_{{2}}.
\end{align*}
The cofactor $C_5=\frac{K_1K_4}{D^2}$  gives the Darboux polynomials $p_{5,1},p_{5,2}$:
\begin{align*}
p_{5,1} &=  \left( x_{{1}}+3\,x_{{2}} \right)  \left( x_{{1}}+x_{{2}} \right),  \\
p_{5,2} &= \left( 5\,hx_{{2}}+2\,hx_{{3}}+8\,h+1 \right)  \left( 5\,hx_{{2}}+2\,
hx_{{3}}+8\,h-1 \right),
\end{align*}
and the  cofactor $C_6=\frac{K_2K_3}{D^2}$  gives the Darboux polynomials $p_{6,1},p_{6,2}$:
\begin{align*}
p_{6,1} &=  - \left( 3\,hx_{{1}}+6\,hx_{{2}}+1 \right)  \left( 3\,hx_{{1}}+6\,hx_{
{2}}-1 \right),  \\
p_{6,2} &= (x_{{2}}+x_{{3}}-1)(4\,x_{{2}}+x_{{3}}+9).
\end{align*}
From these we obtain that the preserved integrals of the Kahan map are 
$ \frac{p_{5,1}({\bf x})}{p_{5,2}({\bf x})}$, $\frac{p_{6,2}({\bf x})}{p_{6,1}({\bf x})} $, 
and any combination
\begin{equation*} 
 \frac{1}{p_{5,i}({\bf x})p_{6,j}({\bf x})} d{\bf x},  \qquad  i,j\in\{1,2\},
\end{equation*}
is a preserved measure.

\subsection{Example 3: Quartic Nahm system in 2D} 
This subsection exemplifies the Kahan discretization of a certain class of ODEs with quartic Hamiltonians.
We consider the following example whose Kahan discretization was studied in \cite{PPS}
\begin{equation}
\frac{d}{dt}\left( \begin{array}{c} x_1 \\ x_2  \end{array} \right) =
\left( \begin{array}{c} 2x_1^2 - 12x_2^2 \\ -6x_1x_2 - 4x_2^2  \end{array} \right). 
\end{equation}
This ODE has a preserved integral
\begin{equation*}
H := x_2(2x_1 + 3x_2 )(x_1-x_2)^2, 
\end{equation*}
and a preserved measure
\begin{equation*}
 \int \frac{dx_1\wedge dx_2}{x_2(2x_1 + 3x_2)(x_1-x_2)}.
\end{equation*}
The Jacobian determinant of the Kahan discretization has the following factorization 
\begin{equation*}
J =\frac{K_1\, K_2\, K_3}{D^3}, \nonumber
\end{equation*}
where the three affine $K_i$ are given by
\begin{align*}
K_1 &:= 1 + 3hx_1 - 8hx_2, \\
K_2 &:= 1 - 5hx_1, \\
K_3 &:= 1 + 3hx_1 + 12hx_2,
\end{align*}
and the quadratic $D$ is
\begin{equation*}
D := 1 + hx_1 + 4hx_2 - 6h^2x_1^2 - 8h^2x_1x_2 - 36h^2x_2^2.  
\end{equation*}
Among the cofactors $K_1^iK_2^jK_3^k/D^l$ for $i,j,k=0,1$ and $l=1,\dots,3$ we consider $C_1 = \frac{K_1}{D}$, $C_2 = \frac{K_2}{D}$, and $C_3 = \frac{K_3}{D}$, satisfying $J = C_1C_2C_3$. The corresponding Darboux polynomials are
\begin{align*}
p_{1,1} &= 2x_1 + 3x_2, \\
p_{2,1} &= x_2, \\ 
p_{3,1} &= x_2 - x_1,   
\end{align*}
leading to the preserved measure
 \begin{equation*}
\frac{dx_1\wedge dx_2 \wedge dx_3}{p_{1,1}p_{2,1}p_{3,1}}.
\end{equation*}
To find the modified integral, we search for Darboux polynomials whose cofactors are of the form  $C_1^iC_2^j$ for $i,j=1,2,...$ (i.e., ``super-factors" of $J$). Using the cofactor 
\begin{equation*}
C_4 := C_1 C_2 C_3^2,
\end{equation*}
we find
\begin{align*}
p_{4,1} &= x_2(2x_1 + 3x_2)(x_1-x_2)^2, \\
p_{4,2} &= (1+h(3x_1+2x_2)) (1-h(3x_1+2x_2)) (1+h(6x_2-x_1)) (1-h(6x_2-x_1)),
\end{align*}
and $\frac{p_{4,1}}{p_{4,2}}$  is an integral of the Kahan discretization,
see the corresponding example in \cite{PPS}.

  Note that the invariant set $p_{4,2}(\bx)=0$, represented by a product of four lines, is shown in red in the phase plot of the Quartic Nahm system in Figure 3 of our paper \cite{VCMMOQ}, for timestep $h=1/5.$

\subsection{Example 4: Lagrange top}
This subsection describes the computation of the modified Darboux polynomials of the Kahan map of one of the classical integrable tops, requiring significant computing power. 
Discretizations of the Lagrange top have been studied in \cite{KH} and \cite{PPS}, where it was shown that the Kahan map preserves a number of modified integrals. The Lagrange top reads
\begin{equation}\label{LagrangeTop}
\frac{d}{dt}\left( \begin{array}{c}
m_1 \\ 
m_2 \\ 
m_3 \\ 
p_1 \\ 
p_2 \\ 
p_3 \\   
\end{array} \right) 
=
\left( \begin{array}{c} 
\left( \alpha-1 \right) m_{{2}}m_{{3}}+\gamma\,p_{{2}}\\
\left( 1-\alpha \right) m_{{1}}m_{{3}}-\gamma\,p_{{1}}\\
0\\
\alpha\,p_{{2}}m_{{3}}-p_{{3}}m_{{2}}\\
p_{{3}}m_{{1}}- \alpha\,p_{{1}}m_{{3}}\\
p_{{1}}m_{{2}}-p_{{2}}m_{{1}}\\
\end{array} \right) 
,
\end{equation}
where $m_i$ and $p_i$ are the angular and linear momentum components and $\alpha$ and $\gamma$ are constant parameters. The Lagrange top admits four independent integrals 
\begin{align*}
	\tilde{H}_1 &=  p_1^2+p_2^2+p_3^2, \\
	\tilde{H}_2 &=  p_1m_1+p_2m_2+p_3m_3, \\
	\tilde{H}_3 &=  m_1^2+m_2^2+\alpha m_3^2 + 2 \gamma p_3,\\
	\tilde{H}_4 &=  m_3.
\end{align*}
For the Lagrange top, it suffices to treat $m_3$ as a free parameter by working in the variables $\mathbf{\bar{x}}=(m_1,m_2,p_1,p_2,p_3)^\mathrm{T}$ and look for degree-six Darboux polynomial densities in $\mathbf{\bar{x}}$. The Jacobian determinant of the Kahan map has the form $$J=\frac{K_1K_2}{D^3},$$ where $K_1$ has 377 terms and $K_2$ has 35 terms. Using the cofactor $C_1({\bf x}) = J$, we find the following five Darboux polynomial densities 
\begin{align*}
	p_{1,1} &=-256\,\gamma+64\,{h}^{2}\gamma\, \left( -2\,{m_{{3}}}^{2}{\alpha}^{2}+2\,{m_{{3}}}^{2}\alpha+\gamma\,p_{{3}}-{m_{{1}}}^{2}-{m_{{2}}}^{2}-{m_{{3}}}^{2} \right)\nonumber\\\qquad & +h^{4}Q_{1,2}^{(4)} +h^{6}Q_{1,3}^{(6)} +h^{8}Q_{1,4}^{(8)} ,
	\\p_{1,2} &=-2048\,{\gamma}^{3}+256\,{h}^{2}{\gamma}^{3} \left( -2\,{m_{{3}}}^{2}{\alpha}^{2}+2\,{m_{{3}}}^{2}\alpha+4\,\gamma\,p_{{3}}-{m_{{1}}}^{2}-{m_{{2}}}^{2}-2\,{m_{{3}}}^{2} \right) \nonumber\\\qquad &+h^{4}Q_{2,2}^{(4)} +h^{6}Q_{2,3}^{(6)} +h^{8}Q_{2,4}^{(8)} +h^{10}Q_{2,5}^{(8)} , \\
p_{1,3}     &=-2048\,m_{{3}} \left( 6\,\alpha-5 \right) +{h}^{2} (1536\,m_3\,(1-\alpha)(m_1^2+m_2^2) +6144\,\alpha\,\gamma\,m_{{3}}p_{{3}} \\
		&+1024\,m_3^3(1-4\,\alpha+10\,\alpha^2-8\,\alpha^3)  -512\,\gamma\,p_{{1}}m_{{1}}-512\,\gamma\,p_{{2}}m_{{2}} \\
		&-5120\,\gamma\,m_{{3}}p_{{3}} ) +h^{4}Q_{3,2}^{(5)} +h^{6}Q_{3,3}^{(7)} +h^{8}Q_{3,4}^{(9)} +h^{10}Q_{3,5}^{(11)} +h^{12}Q_{3,6}^{(11)} ,\\ 
p_{1,4}      &=-256\,\gamma\,m_{{3}} \left( 2\,\alpha-1 \right) +64\,{h}^{2}\gamma\, ( -2\,{\alpha}^{3}{m_{{3}}}^{3}+3\,{\alpha}^{2}{m_{{3}}}^{3}+4\,\alpha\,\gamma\,m_{{3}}p_{{3}}				\\
	        &-\alpha\,{m_{{1}}}^{2}m_{{3}}-\alpha\,{m_{{2}}}^{2}m_{{3}}
	-3\,\alpha\,{m_{{3}}}^{3}+\gamma\,p_{{1}}m_{{1}}+\gamma\,p_{{2}}m_{{2}}-2\,\gamma\,m_{{3}}p_{{3}}+{m_{{3}}}^{3} )\\
	& +h^{4}Q_{4,2}^{(5)} +h^{6}Q_{4,3}^{(7)} +h^{8}Q_{4,4}^{(9)} ,\\
p_{1,5}      &=-65536+  {h}^{2}\big( -32768\,{m_{{3}}}^{2}{\alpha}^{2}+40960\,{m_{{3}}}^{2}\alpha+16384\,(\gamma\,p_{{3}}-\,{m_{{1}}}^{2}-\,{m_{{2}}}^{2}\\\qquad &-2\,{m_{{3}}}^{2})) +h^{4}Q_{5,2}^{(4)} +h^{6}Q_{5,3}^{(6)} +h^{8}Q_{5,4}^{(8)} +h^{10}Q_{5,5}^{(10)} +h^{12}Q_{5,6}^{(12)} +h^{14}Q_{5,7}^{(12)} ,
\end{align*}
where each $Q^{(i)}_{j,k}$ is a polynomial of degree $i$ in the variables $(m_1,m_2,m_3,p_1,p_2,p_3)$. Taking the quotients $\frac{p_{1,1}}{p_{1,5}}$, $\frac{p_{1,2}}{p_{1,5}}$, $\frac{p_{1,3}}{p_{1,5}}$ and $\frac{p_{1,4}}{p_{1,5}}$ yields four functionally independent integrals. Taking functionally dependent combinations of these, we are able to form the following integrals that are preserved by the Kahan discretization
\begin{align*}
	H_1 &=  \frac{p_1^2+p_2^2+p_3^2+\mathcal{O}(h^2)}{1+\mathcal{O}(h^2)}, \\
	H_2 &=  \frac{p_1m_1+p_2m_2+p_3m_3+\mathcal{O}(h^2)}{1+\mathcal{O}(h^2)},\\
	H_3 &=  \frac{m_1^2+m_2^2+\alpha m_3^2 + 2 \gamma p_3+\mathcal{O}(h^2)}{1+\mathcal{O}(h^2)},\\
	H_4 &=  m_3, 
\end{align*}
where the first three integrals are modified versions of the continuous integrals.

\subsection{Example 5: A Kahan map having a non-rational integral}
This example was constructed to display a non-rational integral, an exact linearization and solution, as well as a large number of affine Darboux polynomials (even more when the parameter $\alpha$ equals 1).
\begin{equation}  \label{ODE}
\begin{aligned}
\frac{ d x_1}{ d t} &= 24( x_{2}- x_{4})(1-\alpha ) + 9 x_{1}^{2}+48 x_{1} x_{3}-40 x_{1} x_{4}+24 x_{2}^{2}-48 x_{2} x_{3}\\
       & +48 x_{2} x_{4}+48 x_{3}^{2}+24 x_{3} x_{4}-132 x_{4}^{2}+x_{1}, \\
\frac{ d x_2}{ d t} &=7 (x_{2} - x_{4}) \alpha -2 x_{1}^{2}-12 x_{1} x_{3}+12 x_{1} x_{4}-5 x_{2}^{2}+12 x_{2} x_{3}-14 x_{2} x_{4}\\
      &-12 x_{3}^{2}-6 x_{3} x_{4}+38 x_{4}^{2}-6 x_{2}+7 x_{4},	\\
\frac{ d x_3}{ d t} &= 14 (x_{2}  - x_{4})( \alpha - 1) -4 x_{1}^{2}-24 x_{1} x_{3}+24 x_{1} x_{4}-14 x_{2}^{2}+28 x_{2} x_{3}\\
	&-28 x_{2} x_{4}-25 x_{3}^{2}-12 x_{3} x_{4}+76 x_{4}^{2}+x_{3},	\\
\frac{ d x_4}{ d t} &= 	6( x_{2} - x_{4}) \alpha -2 x_{1}^{2}-12 x_{1} x_{3}+12 x_{1} x_{4}-6 x_{2}^{2}+12 x_{2} x_{3}-12 x_{2} x_{4}\\
	&-12 x_{3}^{2}-6 x_{3} x_{4}+37 x_{4}^{2}-6 x_{2}+7 x_{4},
\end{aligned}
\end{equation}

%
 where $\alpha$ is a parameter.
	
  The Jacobian determinant of the Kahan discretization of this ODE is
\begin{equation}
J={\frac{{K_{{1}}}^{3}{K_{{2}}}^{3}K_{3}K_{4}}{{D_{{1}}}^{2}{D_{{2}}}^{2}{D_{{3}}}^{2}{D_{{4}}}^{2}}},
\end{equation}
where $K_1,\dots,K_4$ are constant, and $D_1,\dots,D_4$ are affine: 
\begin{align*}
K_1 &=  1-\tfrac12 h,
\\
K_2 &=  1+\tfrac12 h,
\\
K_3 &=1 - \tfrac12 \alpha h,
\\
K_4 &= 1+ \tfrac12 \alpha  h,
\\
D_1 &=  1 -\tfrac{1}{2} h -h x_{1}-4 h x_{4},
\\
D_2 &=  1 -\tfrac{1}{2} h  -2 h x_{2}+h x_{3},
\\
D_3 &=  1 -\tfrac{1}{2} h  -2 h x_{1}-3 h x_{3}-h x_{4},
\\
D_4 &= 1 -\tfrac{1}{2} \alpha  h -h x_{2}+h x_{4}.
\end{align*}
The following 11 affine discrete Darboux polynomials $p_{i,1}$ for $i=1,...,11$ are found corresponding to the cofactors $C_i$
\begin{equation}\label{table}
\begin{array}{c|c|c}
i&p_{i,1} & C_i \\
\hline
1 & 1+x_{1}+4 x_{4}
& K_1/D_1
\\
2 & x_{{1}}+4\,x_{{4}}
& K_2/D_1
\\
3 & 1+2 x_{2}-x_{3}
& K_1/D_2
\\
4 & 2 x_{2}-x_{3}
& K_2/D_2
\\
5 & 1+2 x_{1}+3 x_{3}+x_{4}
& K_1/D_3
\\
6 & 2 x_{1}+3 x_{3}+x_{4}
& K_2/D_3
\\
7 & \alpha +x_{2}-x_{4}
& K_3/D_4
\\
8 & x_{2}-x_{4}
& K_4/D_4
\\
9 & x_{{1}}-2\,x_{{2}}+x_{{3}}+4\,x_{{4}}
& K_1K_2/(D_1 D_2)
\\
10 & x_{1} +3 x_{3} -3 x_{4}
& K_1K_2/(D_1 D_3)
\\
11 & 2 x_{1} -2 x_{2} +4 x_{3}+x_{4}
& K_1K_2/(D_2 D_3)
\end{array}
\end{equation}
 It can be read off from table(\ref{table}) that the three ratios $\frac{p_{1,1}}{p_{2,1}}, \frac{p_{3,1}}{p_{4,1}}$ and $ \frac{p_{5,1}}{p_{6,1}}$ each have constant cofactor $K_1/K_2$, and that $ \frac{p_{7,1}}{p_{8,1}}$ has constant cofactor $K_3/K_4$.

  It follows that the three functions
\begin{align} 
H_1 &= \frac{p_{1,1}p_{4,1}}{p_{2,1}p_{3,1}} = \frac{\left(1+x_{1}+4 x_{4}\right) \left(2 x_{2}-x_{3}\right)}{\left(x_{1}+4 x_{4}\right) \left(1+2 x_{2}-x_{3}\right)} \label{eq:H1},  \\
H_2 &= \frac{p_{1,1}p_{6,1}}{p_{2,1}p_{5,1}} = \frac{\left(1+x_{1}+4 x_{4}\right) \left(2 x_{1}+3 x_{3}+x_{4}\right)}{\left(x_{1}+4 x_{4}\right) \left(1+2 x_{1}+3 x_{3}+x_{4}\right)}   \label{eq:H2},  \\
H_3 &= \frac{p_{1,1}p_{8,1}^a}{p_{2,1}p_{7,1}^a} = \frac{\left(1+x_{1}+4 x_{4}\right) \left(x_{2}-x_{4}\right)^{a}}{\left(x_{1}+4 x_{4}\right) \left(\alpha +x_{2}-x_{4}\right)^{a}}    \label{eq:H3},
\end{align}
are all integrals of the Kahan map of (\ref{ODE}), where $a:= \frac{\ln(K_1/K_2)}{\ln(K_3/K_4)}$.

  We remark that in general the integral \eqref{eq:H3} is non-rational, even though  the leaves  $p_{1,1}/p_{2,1} = K_1/K_2$ and 
$p_{7,1}/p_{8,1} = K_3/K_4$ making up the integral are polynomial. \footnote{Note that in the special case $\alpha=1$, $K_3=K_1$ and $K_4=K_2$; the Kahan map of (\ref{ODE}) has 3 additional affine Darboux polynomials; and the integral  \eqref{eq:H3} becomes rational.}

  We also see that
\begin{equation}
C_7C_8C_9C_{10}C_{11} = \frac{K_1^3K_2^3K_3K_4}{D_1^2D_2^2D_3^2D_4^2} = J.
\end{equation}
Hence the Kahan map of the vector field (\ref{ODE}) has the preserved measure
\begin{equation} \label{dens}
\frac{d \bx}{p_{7,1}p_{8,1}p_{9,1}p_{10,1}p_{11,1}} ,
\end{equation}
and therefore the Kahan map is super-integrable. It follows that the ODE (\ref{ODE}) is also super-integrable, preserving the same integrals and measure as the map.

Finally, it can be read off from Table(\ref{table})  and the fact that $K_1,\dots,K_4$ are constant that
\begin{equation}  \label{rats}
\begin{aligned}
p_{1,1}(\bx_n) &= \beta_1 p_{2,1}(\bx_n)(K_1/K_2)^n,  \\
p_{3,1}(\bx_n) &= \beta_2 p_{4,1}(\bx_n)(K_1/K_2)^n, \\
p_{5,1}(\bx_n) &= \beta_3 p_{6,1}(\bx_n)(K_1/K_2)^n,  \\
p_{7,1}(\bx_n) &= \beta_4 p_{8,1}(\bx_n)(K_3/K_4)^n, 
\end{aligned}
\end{equation}
where the $\beta_i$ are integration constants.

  The equations (\ref{rats}) represent 4 linear equations in the four variables $x_1,x_2,x_3,x_4$. Solving them, we obtain the exact solution of  the Kahan map of eq(\ref{ODE}):

\begin{equation}
\begin{aligned}
\label{exsol}
x_1(n) &= \frac{1}{D} ( (4 ((\beta_{1}-\frac{7 \beta_{3}}{4}) \beta_{2}+3 \beta_{1} \beta_{3}) \beta_{4} \mathit{r_2^n}+24(( \beta_{1}+ \beta_{3}) \beta_{2}+ \beta_{1} \beta_{3}) \alpha    \\
& +(-4 \beta_{1}+7 \beta_{3}) \beta_{2}-12 \beta_{1} \beta_{3}) \mathit{r_1^{2 n}}+ (9-(16\beta_{1}- 3 \beta_{2} + 5 \beta_{3}) \mathit{r_1^n}) \beta_{4} \mathit{r_2^n} 
 \\
&  +(24(- \beta_{1}- \beta_{2}- \beta_{3}) \alpha +16 \beta_{1}-3 \beta_{2}+5 \beta_{3}) \mathit{r_1^{2n}} \\
	&-24 \mathit{r_1^{3 n}} \alpha  \beta_{1} \beta_{2} \beta_{3}+24 \alpha -9  ),  \\
x_2(n) &= \frac{1}{D} ( (- ( (\beta_{1}-2 \beta_{3}) \beta_{2}+3 \beta_{1} \beta_{3} ) \beta_{4} \mathit{r_2^n}+ ((-7 \beta_{1}-7 \beta_{3}) \beta_{2}-7 \beta_{1} \beta_{3} ) \alpha   \\
&   +(\beta_{1}-2 \beta_{3}) \beta_{2}+3 \beta_{1} \beta_{3}) \mathit{r_1^{2 n}}+4 (-\frac{1}{2}+(\beta_{1}-\frac{\beta_{2}}{4}+\frac{\beta_{3}}{4}) \mathit{r_1^n}) \beta_{4} \mathit{r_2^n} \\
&  +((7 \beta_{1}+7 \beta_{2}+7 \beta_{3}) \alpha -4 \beta_{1}+\beta_{2}-\beta_{3}) \mathit{r_1^n}+7 \mathit{r_1^{3 n}} \alpha  \beta_{1} \beta_{2} \beta_{3}-7 \alpha +2 ),  \\
x_3(n) &= \frac{1}{D} (  (-2 ((\beta_{1}-2 \beta_{3}) \beta_{2}+\frac{7 \beta_{1} \beta_{3}}{2}) \beta_{4} \mathit{r_2^n}-14(( \beta_{1}+ \beta_{3}) \beta_{2}+ \beta_{1} \beta_{3}) \alpha
 \\
&   +(2 \beta_{1}-4 \beta_{3}) \beta_{2}+7 \beta_{1} \beta_{3}) \mathit{r_1^{2 n}}+9 (-\frac{5}{9}+(\beta_{1}-\frac{2 \beta_{2}}{9}+\frac{\beta_{3}}{3}) \mathit{r_1^n}) \beta_{4} \mathit{r_2^n} +5 \\
&  +(14( \beta_{1}+ \beta_{2}+ \beta_{3}) \alpha -9 \beta_{1}+2 \beta_{2}-3 \beta_{3}) \mathit{r_1^n}
	+14\alpha (\mathit{r_1^{3 n}}   \beta_{1} \beta_{2} \beta_{3}-1)  ), \\
x_4(n) &= \frac{1}{D} (  (-((\beta_{1}-2 \beta_{3}) \beta_{2}+3 \beta_{1} \beta_{3}) \beta_{4} \mathit{r_2^n}+((-6 \beta_{1}-6 \beta_{3}) \beta_{2}-6 \beta_{1} \beta_{3}) \alpha   \\
&  +(\beta_{1}-2 \beta_{3}) \beta_{2}+3 \beta_{1} \beta_{3}) \mathit{r_1^{2 n}}+4 (-\frac{1}{2}+(\beta_{1}-\frac{\beta_{2}}{4}+\frac{\beta_{3}}{4}) \mathit{r_1^n}) \beta_{4} \mathit{r_2^n}    \\
& +((6 \beta_{1}+6 \beta_{2}+6 \beta_{3}) \alpha -4 \beta_{1}+\beta_{2}-\beta_{3}) \mathit{r_1^n}+6 \mathit{r_1^{3 n}} \alpha  \beta_{1} \beta_{2} \beta_{3}-6 \alpha +2 ),
\end{aligned}
\end{equation} 
where $$r_1=K_1/K_2, r_2=K_3/K_4, \mbox{ and } D=\left(\beta_{4} \mathit{r_2^n}-1\right) \left(\beta_{3} \mathit{r_1^n}-1\right) \left(\beta_{1} \mathit{r_1^n}-1\right) \left(\beta_{2} \mathit{r_1^n}-1\right).$$


  The exact solution of the ODE (\ref{ODE}) is obtained from (\ref{exsol}) by replacing $\mathit{r_1^n}$ by $e^{t}$, resp $\mathit{r_2^n}$ by $e^{\alpha t}$.

\subsection{Example 6: A 4D polarization map}
This subsection exemplifies the fact that with regard to the discretization of ODEs, the application of our method is not restricted to Kahan's discretization. Here we study an application to our polarization method \cite{CMMOQpol}. For a different application, cf \cite{HQ}.

  We consider the 4-dimensional map presented in \cite{CMMOQpol} (choosing $a=2,b=1,c=-3,$ $d=-1,e=1$ in their notation).
\begin{align*}
x'_1 &= x_2,    \\
x'_2 &= \frac{2\,h{x_{{1}}}^{2}x_{{2}}-6\,h{x_{{1}}}^{2}x_{{4}}-12\,hx_{{1}}x_{{2}}x_{{3}}-4\,hx_{{1}}x_{{3}}x_{{4}}-2\,hx_{{2}}{x_{{3}}}^{2}+2\,h{x_{{3}}}^{2}x_{{4}}+x_{{1}}}{D({\bf x})},  \\
x'_3 &= x_4,    \\
x'_4 &= \frac{-4\,h{x_{{1}}}^{2}x_{{2}}-2\,h{x_{{1}}}^{2}x_{{4}}-4\,hx_{{1}}x_{{2}}x_{{3}}+12\,hx_{{1}}x_{{3}}x_{{4}}+6\,hx_{{2}}{x_{{3}}}^{2}+2\,h{x_{{3}}}^{2}x_{{4}}+x_{{3}}}{D({\bf x})},
\end{align*}
where the quartic $D$ is
\begin{align*}
D &:= -28\,{h}^{2}{x_{{1}}}^{2}{x_{{2}}}^{2}+4\,{h}^{2}{x_{{1}}}^{2}x_{{2}}x
_{{4}}-40\,{h}^{2}{x_{{1}}}^{2}{x_{{4}}}^{2}+4\,{h}^{2}x_{{1}}{x_{{2}}
}^{2}x_{{3}}-28\,{h}^{2}x_{{1}}x_{{2}}x_{{3}}x_{{4}} \\  & -8\,{h}^{2}x_{{1}}
x_{{3}}{x_{{4}}}^{2}-40\,{h}^{2}{x_{{2}}}^{2}{x_{{3}}}^{2}-8\,{h}^{2}x
_{{2}}{x_{{3}}}^{2}x_{{4}}-16\,{h}^{2}{x_{{3}}}^{2}{x_{{4}}}^{2}+1.
\end{align*}
The determinant of the Jacobian of the map has the factorized form
\begin{equation*}
J = \frac{K_1}{D^3},
\end {equation*}
where
\begin{align*}
K_1 &:=  -128\,{h}^{3}{x_{{1}}}^{3}{x_{{2}}}^{3}+456\,{h}^{3}{x_{{1}}}^{3}{x_{{
2}}}^{2}x_{{4}}-120\,{h}^{3}{x_{{1}}}^{3}x_{{2}}{x_{{4}}}^{2}+496\,{h}
^{3}{x_{{1}}}^{3}{x_{{4}}}^{3} \\ & +984\,{h}^{3}{x_{{1}}}^{2}{x_{{2}}}^{3}x
_{{3}}+432\,{h}^{3}{x_{{1}}}^{2}{x_{{2}}}^{2}x_{{3}}x_{{4}}+936\,{h}^{
3}{x_{{1}}}^{2}x_{{2}}x_{{3}}{x_{{4}}}^{2}+336\,{h}^{3}{x_{{1}}}^{2}x_
{{3}}{x_{{4}}}^{3} \\  & +408\,{h}^{3}x_{{1}}{x_{{2}}}^{3}{x_{{3}}}^{2}-864\,
{h}^{3}x_{{1}}{x_{{2}}}^{2}{x_{{3}}}^{2}x_{{4}}+216\,{h}^{3}x_{{1}}x_{
{2}}{x_{{3}}}^{2}{x_{{4}}}^{2}-528\,{h}^{3}x_{{1}}{x_{{3}}}^{2}{x_{{4}
}}^{3} \\ & -488\,{h}^{3}{x_{{2}}}^{3}{x_{{3}}}^{3}-192\,{h}^{3}{x_{{2}}}^{2
}{x_{{3}}}^{3}x_{{4}}-264\,{h}^{3}x_{{2}}{x_{{3}}}^{3}{x_{{4}}}^{2}-80
\,{h}^{3}{x_{{3}}}^{3}{x_{{4}}}^{3} \\ & -84\,{h}^{2}{x_{{1}}}^{2}{x_{{2}}}^
{2}+12\,{h}^{2}{x_{{1}}}^{2}x_{{2}}x_{{4}}-120\,{h}^{2}{x_{{1}}}^{2}{x
_{{4}}}^{2}+12\,{h}^{2}x_{{1}}{x_{{2}}}^{2}x_{{3}}-84\,{h}^{2}x_{{1}}x
_{{2}}x_{{3}}x_{{4}} \\  & -24\,{h}^{2}x_{{1}}x_{{3}}{x_{{4}}}^{2}-120\,{h}^{
2}{x_{{2}}}^{2}{x_{{3}}}^{2}-24\,{h}^{2}x_{{2}}{x_{{3}}}^{2}x_{{4}}-48
\,{h}^{2}{x_{{3}}}^{2}{x_{{4}}}^{2}+1.
\end{align*}
Using $C_1=J$ as cofactor, the resulting functionally independent Darboux polynomials are
\begin{align*}
p_{1,1} &= D, \\
p_{1,2} & = \left( x_{{1}}x_{{4}}-x_{{2}}x_{{3}} \right)  ( 4\,h{x_{{1}}}^{2
}{x_{{2}}}^{2}+4\,h{x_{{1}}}^{2}x_{{2}}x_{{4}}-6\,h{x_{{1}}}^{2}{x_{{4
}}}^{2}+4\,hx_{{1}}{x_{{2}}}^{2}x_{{3}} -x_{{2}}x_{{3}} \\
	& -24\,hx_{{1}}x_{{2}}x_{{3}}x_{{
4}}   -4\,hx_{{1}}x_{{3}}{x_{{4}}}^{2}-6\,h{x_{{2}}}^{2}{x_{{3}}}^{2}-4\,
hx_{{2}}{x_{{3}}}^{2}x_{{4}}+2\,h{x_{{3}}}^{2}{x_{{4}}}^{2}+x_{1}x_{
4}),   \\
p_{1,3} &= 26\,h{x_{{1}}}^{3}{x_{{2}}}^{2}x_{{4}}+10\,h{x_{{1}}}^{3}x_{{2}}{x_{{4
}}}^{2}+2\,h{x_{{1}}}^{3}{x_{{4}}}^{3}-26\,h{x_{{1}}}^{2}{x_{{2}}}^{3}
x_{{3}}-60\,h{x_{{1}}}^{2}x_{{2}}x_{{3}}{x_{{4}}}^{2} \\  & -8\,h{x_{{1}}}^{2
}x_{{3}}{x_{{4}}}^{3}-10\,hx_{{1}}{x_{{2}}}^{3}{x_{{3}}}^{2}+60\,hx_{{
1}}{x_{{2}}}^{2}{x_{{3}}}^{2}x_{{4}}+14\,hx_{{1}}{x_{{3}}}^{2}{x_{{4}}
}^{3}-2\,h{x_{{2}}}^{3}{x_{{3}}}^{3} \\  & +8\,h{x_{{2}}}^{2}{x_{{3}}}^{3}x_{
{4}}-14\,hx_{{2}}{x_{{3}}}^{3}{x_{{4}}}^{2}+2\,{x_{{1}}}^{2}{x_{{2}}}^
{2}+2\,{x_{{1}}}^{2}x_{{2}}x_{{4}}+2\,x_{{1}}{x_{{2}}}^{2}x_{{3}} \\
	&-18\,x_{{1}}x_{{2}}x_{{3}}x_{{4}}  -2\,x_{{1}}x_{{3}}{x_{{4}}}^{2}-2\,x_{{2}}
{x_{{3}}}^{2}x_{{4}}+{x_{{3}}}^{2}{x_{{4}}}^{2}.
\end{align*}
The map thus possesses the preserved measure $\int \frac{d \bx}{p_{1,1}}$ and the (independent) first integrals  $\frac{p_{1,2}}{p_{1,1}}$ and $\frac{p_{1,3}}{p_{1,1}}$, in agreement with \cite{CMMOQpol}.

\subsection{Example 7: sine-Gordon maps}
This subsection exemplifies the application of our method to maps arising from the theory of discrete integrable systems. We consider $(k+1)$-dimensional maps that arise as so-called $(1,k)$ reductions of the discrete sine-Gordon equation \cite{quispel91imd}.

  We start with the case $k=3$, then treat the case $k=2$, before giving a general theorem for arbitrary $k$. 
\subsubsection{The $(1,3)$ sine-Gordon map.}

The $(1,3)$ sine-Gordon map $\phi$ is given by
\begin{align*}
x_i' &= x_{i+1}, \qquad i=0,1,2, \\
x_3' &= \frac{1-\alpha x_1 x_3}{x_0(x_1 x_3 - \alpha)},
\end{align*}
where $\alpha$ is a parameter.
  Using $C_1({\bf x}) = J$, we find the corresponding Darboux polynomials:
\begin{align*}
p_{1,1} &= x_{{3}}x_{{2}}x_{{1}}x_{{0}}, \\
p_{1,2} &= {x_{{0}}}^{2}x_{{1}}x_{{2}}{x_{{3}}}^{2}-\alpha\,{x_{{0}}}^{2}x_{{2}}x
_{{3}}-\alpha\,x_{{0}}{x_{{1}}}^{2}x_{{3}}-\alpha\,x_{{0}}x_{{1}}{x_{{
2}}}^{2}  \\
&  -\alpha\,x_{{0}}x_{{1}}{x_{{3}}}^{2}-\alpha\,x_{{0}}{x_{{2}}}^
{2}x_{{3}}-\alpha\,{x_{{1}}}^{2}x_{{2}}x_{{3}}+x_{{1}}x_{{2}} , \\
p_{1,3} &= {x_{{0}}}^{2}{x_{{1}}}^{2}x_{{2}}x_{{3}}+x_{{0}}{x_{{1}}}^{2}{x_{{2}}}
^{2}x_{{3}}+x_{{0}}x_{{1}}{x_{{2}}}^{2}{x_{{3}}}^{2}-\alpha\,{x_{{0}}}
^{2}x_{{1}}x_{{2}}  \\  & -\alpha\,x_{{1}}x_{{2}}{x_{{3}}}^{2}+x_{{0}}x_{{1}}+
x_{{0}}x_{{3}}+x_{{3}}x_{{2}}.
\end{align*}
It follows that $\phi$ possesses the (independent) first integrals $\frac{p_{1,2}}{p_{1,1}}$ and $\frac{p_{1,3}}{p_{1,1}}$, and the preserved measure $\int \frac{d\bx}{p_{1,1}}$. These results were found using different methods in  \cite{quispel91imd}.

\subsubsection{The $(1,2)$ sine-Gordon map.}

The $(1,2)$ sine-Gordon map $\phi$ is given by
\begin{align*}
x_i' &= x_{i+1}, \qquad i=0,1, \\
x_2' &= \frac{1-\alpha x_1 x_2}{x_0(x_1 x_2 - \alpha)}.
\end{align*}
Using $C_1({\bf x}) = - J$, we find the corresponding Darboux polynomials:
\begin{align*}
p_{1,1} &= x_{{0}}x_{{1}}x_{{2}}, \\
p_{1,2} &=  {x_{{0}}}^{2}x_{{1}}{x_{{2}}}^{2}-\alpha\,{x_{{0}}}^{2}x_{{2}}-\alpha
\,x_{{0}}{x_{{1}}}^{2}-\alpha\,x_{{0}}{x_{{2}}}^{2}-\alpha\,{x_{{1}}}^
{2}x_{{2}}+x_{{1}},  \\
p_{1,3} &= {x_{{0}}}^{2}{x_{{1}}}^{2}x_{{2}}+x_{{0}}{x_{{1}}}^{2}{x_{{2}}}^{2}-
\alpha\,{x_{{0}}}^{2}x_{{1}}-\alpha\,x_{{1}}{x_{{2}}}^{2}+x_{{0}}+x_{{
2}}.
\end{align*}
It follows that $\phi$ possesses the (independent) first integrals  $\frac{p_{1,2}}{p_{1,1}}$ and $\frac{p_{1,3}}{p_{1,1}}$, and the preserved measure $\int \frac{d \bx}{p_{1,1}}$. Note that this is one extra first integral, that was not found using the Lax representation approach of ref \cite{quispel91imd}. Moreover, we find that there is an additional cofactor $C_2({\bf x}) = J$, for which we find the corresponding Darboux polynomial 
$$p_{2,1} =  -{x_{{0}}}^{2}x_{{1}}{x_{{2}}}^{2}+\alpha\,{x_{{0}}}^{2}x_{{2}}-\alpha
\,x_{{0}}{x_{{1}}}^{2}+\alpha\,x_{{0}}{x_{{2}}}^{2}-\alpha\,{x_{{1}}}^
{2}x_{{2}}+x_{{1}} .$$
Normally, a sole Darboux polynomial does not yield an integral, but, because $C_2 = -C_1$, we find that
$$ 
	H := \frac{p_{2,1}}{p_{1,1}}, 
$$
is a so-called 2-integral \cite{HBQC}, i.e. an integral of $\phi \circ \phi$. In this case $H({\bf x}') = -H({\bf x})$. This result was not found using the Lax matrix approach in \cite{quispel91imd}. 

\subsubsection{The $(1,k)$ sine-Gordon map.} 
  The $(1,k)$ sine-Gordon map $\phi_k$ is given by
\begin{equation}\label{sG1k}
\begin{array}{lcl}
x_i' &=& x_{i+1}, \qquad i=0, \dots,k-1, \\
    x_k' &=& \displaystyle{ \frac{1-\alpha x_1 x_k}{x_0(x_1 x_k - \alpha)} },
\end{array}
\end{equation}
where $\lfloor \frac{k+1}{2} \rfloor$ functionally independent rational integrals for this map were found using a Lax matrix approach in \cite{quispel91imd}.
Denote these integrals by $H_k^n({\bf x}) = \frac{N_k^n({\bf x})}{D_k^n({\bf x})}$, $n=1,\dots, \lfloor \frac{k+1}{2} \rfloor$, and define
\begin{equation}\label{define_eps}
\epsilon:= (-1)^{k+1}.
\end{equation}

\begin{theorem}
For all $n$ and $k$, the Darboux polynomials $N_k^n$ and $D_k^n$ are given by 
\begin{align*}
N_k^n({\bf x'}) &= C({\bf x}) N_k^n({\bf x}), \\
D_k^n({\bf x'}) &= C({\bf x}) D_k^n({\bf x}),
\end{align*}
where the cofactor $C({\bf x})$ depends only on $k$, and is given by $C({\bf x}) = \epsilon |D\phi_k({\bf x})|$.
\end{theorem}
The proof is given in appendix A.

\section{Detecting Darboux polynomials and integrals} \label{detect}

 Given a rational map $\mathbf{x}'=\phi({\bf x}):\mathbb{R}^n\rightarrow\mathbb{R}^n$ containing $k$ free parameters denoted by $\boldsymbol{\alpha}=(\alpha_1,\alpha_2,\dots,\alpha_k)$, one could ask if there exist particular choices of $\boldsymbol{\alpha}$ such that $\phi$ preserves additional second integrals.  We note that, in contrast to the {\it linear} cofactor equation (\ref{cofeqn}), this amounts to solving the {\it nonlinear} cofactor equation
 \begin{equation}\label{cofeq}
 	p(\mathbf{x}') = C(\mathbf{x};\boldsymbol{\alpha}) p(\mathbf{x})
 \end{equation}
for the Darboux polynomial indeterminants as well as the parameters, where $C(\mathbf{x};\boldsymbol{\alpha})$ can be non-linear in $\boldsymbol{\alpha}$. 

 \subsection{Example 8: Extended McMillan map}
 Consider the following rational map $\phi({\bf x})$ defined by
\begin{equation*}
	\phi\left(\begin{array}{c}
		x_1\\x_2
	\end{array}\right) = \left(\begin{array}{c}
	-x_2 -f(x_1)\\x_1
\end{array}\right),
\end{equation*}
where 
\begin{equation*}
	f(x_1) = {\frac {\alpha_1\,{x_{{1}}}^{3}+\alpha_2\,{x_{{1}}}^{2}+\alpha_3\,x_{{
					1}}+\alpha_4}{\alpha_5\,{x_{{1}}}^{2}+\alpha_2\,x_{{1}}+\alpha_6}},
\end{equation*}
and $\boldsymbol{\alpha}=(\alpha_1,\dots,\alpha_6)$ are free parameters. The integrability of a special case of this map was studied in \cite{RV}. The Jacobian of the map $\phi$ is $J=1$. If all parameters $\boldsymbol{\alpha}$ are arbitrary and using $C=J$, the equation
\begin{equation}
	p(\phi({\bf x})) = p({\bf x}), \label{phieq}
\end{equation} 
has only one solution $p_1({\bf x})=1$. Solving the non-linear cofactor equation (\ref{cofeq}) yields the condition $\alpha_1=0$. This is an integrable map known as the McMillan map \cite{McM}. Enforcing this condition, one now finds two solutions to equation \eqref{phieq}
\begin{align*}
	p_{{1}}({\bf x})
	&= 1,
	\\
	p_{{2}}({\bf x})
	&= \alpha_5\,{x_{{1}}}^{2}{x_{{2}}}^{2}+\alpha_2\left(\,{x_{{1}}}^{2}x_{{2}}+\,
	x_{{1}}{x_{{2}}}^{2}\right)+\alpha_3\,x_{{1}}x_{{2}}+\alpha_6\left(\,{x_{{1}}}^{2}+
	\,{x_{{2}}}^{2}\right)\nonumber\\&\quad+\alpha_4\left(\,x_{{1}}+\,x_{{2}}\right),
\end{align*}
where $p_2({\bf x})$ is a preserved integral of $\phi$, in agreement with McMillan. 

\subsection{Example 9: Two coupled Euler tops}
This subsection exemplifies an ODE that possesses a non-rational integral. This integral only seems to be preserved by the Kahan map in two special cases: one rational and one polynomial. A Kahan map that possesses a non-rational integral, that becomes rational in the continuum limit, is given in example 5.

  We now consider two coupled Euler tops whose vector field is given by

\begin{equation}\label{cetvf}
\frac{d}{dt}\left( \begin{array}{c}
x_1 \\ 
x_2 \\ 
x_3 \\ 
x_4 \\ 
x_5 \\   
\end{array} \right) 
=
\left( \begin{array}{c} 
a_1^2x_2x_3\\
 a_2^2x_3x_1 \\
 a_3^2x_1x_2 +  a_4^2x_4x_5\\
 a_5^2x_5x_3\\
 a_6^2x_3x_4
\end{array} \right). 
\end{equation}
This system was first presented in \cite{GMN}, and its integrals after discretisation were first explored in \cite{PPS}, where the authors present the following three independent integrals of motion
\begin{equation*}
H_1 = a^2_2 x_1^2 - a_1^2 x_2^2,\quad	
H_2 = a^2_3a^2_5x_2^2-a^2_2a^2_5x_3^2+a^2_2a^2_4x_4^2, \quad H_3 = a^2_6 x_4^2 - a^2_5 x_5^2, 	
\end{equation*}
however we report here the existence of a fourth independent integral given by 
\begin{equation*}
	H_4 = \frac {\left( a_{{1}}x_{{2}}+ a_{{2}
		}x_{{1}} \right) ^{{ {a_{{5}}a_{{6}}}}}}{\left(a_{{5}}x_{{5}}+ a_{{6}}x_{{4}}\right)^{a_{{1}}a_{{2}}}},
\end{equation*}
hence the system is super-integrable. To our knowledge, the integral $H_4$ is new. The Jacobian determinant of the Kahan map has the following factors 
\begin{equation*}
J =\frac{K_1\, K_2\, K_3\, K_4\, K_5}{D^6}. \nonumber
\end{equation*} 
The cofactors $C_i=\frac{K_i}{D}$, for $i=1,2,4,5$ admit the following linear Darboux polynomials 
\begin{align*}
p_{1,1} &=  a_{{5}}x_{{5}}+a_{{6}}x_{{4}},
\\ p_{2,1} &= a_{{5}}x_{{5}}-a_{{6}}x_{{4}},
\\ p_{4,1} &= a_{{1}}x_{{2}}+a_{{2}}x_{{1}},
\\ p_{5,1} &= a_{{1}}x_{{2}}-a_{{2}}x_{{1}},
\end{align*}
however, the cofactor $C_3$ admits no polynomial solutions, up to Darboux polynomials of degree $6$. Now we look for quadratic Darboux polynomials with the cofactors $C_6:=C_1C_2$ and $C_7:=C_4C_5$ and get the following
\begin{align}
	p_{6,1} &= p_{1,1}p_{2,1}, &	p_{6,2} &= (2-ha_5a_6x_3)(2+ha_5a_6x_3) ,\label{pCET1}\\
	p_{7,1} &= p_{4,1}p_{5,1}, &	p_{7,2} &=(2-ha_1a_2x_3)(2+ha_1a_2x_3) .\label{pCET2}
\end{align}
We note that $p_{6,2}$ and $p_{7,2}$ also factorise. In the ODE case, if a Darboux polynomial factorises, each factor is also a Darboux polynomial. In the discrete case that need not be the case, and indeed it often is not true. (To our knowledge, this possibility was first raised in Gasull and Manosa \cite{GM}). Here for instance we have that $p_{6,2}$ and $p_{7,2}$ factorize, i.e. 
\begin{equation*}
	p_{6,2} = q_{6,1}q_{6,2},\quad p_{7,2} = q_{7,1}q_{7,2},
\end{equation*}
where the $q_{i,j}$ satisfy
\begin{align*}
	q'_{6,1} &= C_1 q_{6,2}, &	 q'_{6,2} &= C_2 q_{6,1} \\
	q'_{7,1} &= C_4 q_{7,2}, &	q'_{7,2} &= C_5 q_{7,1}
\end{align*}
which implies that each $q_{i,j}$ is in fact a discrete Darboux polynomial of the second iterate of the Kahan map. The Darboux polynomials from equations \eqref{pCET1} and \eqref{pCET2} yield two independent integrals $\frac{p_{6,1}}{p_{6,2}}$ and $\frac{p_{7,1}}{p_{7,2}}$, in agreement with \cite{PPS}. We also note that no (Darboux) polynomial measures are found up to degree $6$ using $J$ as the cofactor.

We now attempt to solve the non-linear cofactor equation 
\begin{equation}\label{cet3}
	p(\mathbf{x}')=	C_3(\mathbf{x};\boldsymbol{\alpha})p(\mathbf{x}).
\end{equation} 
For a polynomial basis of degree $2$, equation \eqref{cet3} 
admits three conditions that yield non-trivial Darboux polynomials: $a_3=0$, $a_4=0$ and  $a_1^2a_2^2=a^2_5a^2_6$. The first two correspond to the decoupling of two of the equations and these two less interesting cases have three independent discrete integrals each. The third condition is presented in \cite{PPS}. In this case the Jacobian determinant of the Kahan discretization now factors as
\begin{equation*}
J =\frac{{K_6}^3{K_7}^3}{D^6}.
\end{equation*}
Using $C_8=\frac{K_6K_7}{D^2}$ as the cofactor, we get the following six Darboux polynomials
\begin{align*}
	p_{8,1} &= a^4_2x_1^2-a^2_5a_6^2x_2^2,\\
	p_{8,2} &= a^2_2x_1x_5-a^2_6x_2x_4,\\
	p_{8,3} &= a^2_2x_1x_4-a^2_5x_2x_5,\\
	p_{8,4} &= a^2_2a^2_4x_4^2-a^2_2a^2_5x_3^2+a^2_3a^2_5x_2^2,\\
	p_{8,5} &= a^2_2a^2_4x_5^2-a^2_2a^2_6x_3^2+a^2_3a^2_6x_2^2,\\
	p_{8,6} &= 4-a^2_5a^2_6x_3^2h^2.
\end{align*}
Hence, the following measures are preserved
\begin{equation*}
\int\frac{d\bx}{p_{8,i}p_{8,j}p_{8,k}},
\quad \mathrm{for~any} \quad i,j,k=1,\dots,6,
\end{equation*}
and the following integrals are preserved
\begin{equation*}
\frac{p_{8,i}}{p_{8,k}},\quad \text{for} \quad i\ne k,
\end{equation*} 
of which four are independent. The choice $k=6$ yields the integrals presented in \cite{PPS}. 
\subsection{Example 10: A family of Nambu systems with rational integrals} 
Here we will consider Nambu systems, of the form 
\begin{equation}\label{nambu}
\dot{\mathbf{x}}=c\left(\nabla H_1 \times \nabla H_2\right), \quad \bx \in \mathbb{R}^3,
\end{equation}
where $c=y^{2-\alpha}$, $H_1 =\frac{x}{y}$, $H_2 =y^{\alpha}Q$, where $\alpha$ is a free parameter. Initially we take $Q$ to be 
\begin{equation*}
Q= a_1x^2+a_2y^2+a_3z^2+a_4xy+a_5yz+a_6zx.
\end{equation*}
We get the following Jacobian determinant for the Kahan map 
\begin{equation*}
	J=\frac{K_1^2K_2}{D^4}.
\end{equation*}
The cofactor $C_1:=K_1/D$ has the following two Darboux polynomials at degree one
\begin{equation*}
	p_{1,1} = x ~~ \text{and} ~~ p_{1,2} = y,
\end{equation*}
hence the integral $H_1$ is preserved exactly by the Kahan method. The cofactor $C_2:=K_2/D^2$ has the following Darboux polynomial
\begin{equation*}
	p_{2,1} = Q.
\end{equation*}
Using $C_3 = J$, we find that the Kahan discretisation has three preserved measures corresponding to the densities
\begin{equation*}
	p_{3,1}=x^2\,Q, ~~p_{3,2}=xy\,Q~~\text{and} ~~ p_{3,3}=y^2Q,
\end{equation*}
which yield only one independent integral, $H_1$. 

  In order to facilitate additional computations, from now on we use fixed integer coefficients $\mathbf{a}= (3,5,7,11,13,17)$ for $Q$, so we can use our detection algorithm to search for any values of the parameter $\alpha$ such that the Kahan discretisation yields extra Darboux polynomial solutions. To do this, we solve the non-linear cofactor equation
\begin{equation*}
p(\mathbf{x}')=|D\phi({\bf x})| p(\mathbf{x})
\end{equation*}
for $\alpha$ and the Darboux polynomials $p$ of degree 4. This gives us the following solutions for $\alpha$ and the corresponding additional second integral of the Kahan discretisation:
\begin{center}
	\begin{tabular}{r|l}
		$\alpha$ & Integrals \\
		\hline
		-2 &  $H_1$ and $H_2$ \\
		-1 &  $H_1$ and $H_2$ \\
		0 & $H_1$ and $\tilde{H}_{2,0}$ \\
		1 & $H_1$ and $\tilde{H}_{2,1}$ \\
		2 & $H_1$ and $\tilde{H}_{2,2}$ \\	\end{tabular}
\end{center}
where 
\begin{align*}
	\tilde{H}_{2,0} &= {\frac {Q}{12+{h}^{2}\left( 1853\,xy+3485\,xz+938\,{y}^{2}+2665\,yz+1435\,{z}^{2} \right)}}, \\
	\tilde{H}_{2,1} &={\frac { yQ}{1- {h}^{2}\left( 226\,{x}^{2}+211\,xy+119\,xz+64\,{y}^{2}+91\,yz
			+49\,{z}^{2} \right) }},
	\\
	\tilde{H}_{2,2} &= \frac {y^2Q }{48-{h}^{2} A_2 - h^4\,A_4},	
\end{align*}
and
\begin{align*}
A_2 &=  26616\,{x}^{2}+23472\,xy+11424\,xz+6840\,{y}^{2}+8736\,yz+
		4704\,{z}^{2}, \\
	A_4 &= 6309873\,{x}^{3} z-10784832\,{x}^{2}yz+ 1918455\,{x}^{2}{z}^{2}+27341015\,x{y}^{3}\\
	&+37337147\,x{y}^{2}z-7467243 \,xy{z}^{2}-559776\,x{z}^{3} +14528513\,{y}^{4}\\
	& +37680292\,{y}^{3}z + 19891900\,{y}^{2}{z}^{2}-428064\,y{z}^{3}-115248\,{z}^{4}.\\
\end{align*}
\section{Concluding remarks}

We have proposed an approach based on the Jacobian factor ansatz to search for the preserved measures and integrals of a birational map and applied it to a number of examples. The approach uses Darboux polynomials. We have shown that the method can be used to both determine and detect measures and integrals. Some of the examples have required the use of relatively large computer memory space and computational time. The complexity of the method is composed of those of the following main steps
\begin{enumerate}
\item Compute the Jacobian determinant of the map that is analysed
\item Factor the Jacobian determinant
\item Calculate the corresponding Darboux polynomial
\end{enumerate}
These three items all have well known complexity measures when there are no free parameters. For maps of $\mathbb{R}^n$ they all typically have complexity $\mathcal{O}(n^3)$. However, the situation becomes more difficult to analyse when there are free parameters involved, also because it is not clear to us exactly which algorithms are used internally in Maple or whether the complexity can be quantified in general. 
Experience shows that the computation time in our experiments can have large variations from case to case.
It is also important to notice that the bottle neck for these computations often seems to be memory usage rather than computational complexity.

 In this paper Examples 1,2,3,6,7,8 could all be computed on a laptop/desktop computer, but Examples 4,9, and 10 needed more power and memory and were solved on a supercomputer.


\section*{Acknowledgements}
This work was partially supported by the Australian Research Council, by the Research Council of Norway, by the European Union's Horizon 2020 research and innovation program under the Marie Sk\l{}odowska-Curie grant agreement No. 691070.  The authors would like to thank the Isaac Newton Institute for Mathematical Sciences, Cambridge, for support and hospitality during the programme Geometry, compatibility and structure preservation in computational differential equations (2019) where work on this paper was undertaken, EPSRC grant EP/K032208/1.  This work was also supported by: EPSRC grant number EP/R014604/1. Celledoni and Quispel are also grateful to the Simons Foundation for Fellowships supporting this work. We are indebted to Giorgio Gubbiotti and Peter van der Kamp for useful discussions.

\section*{Appendix A. Proof of Theorem 2.}
In \cite{quispel91imd} it was shown that $\lfloor \frac{k+1}{2}  \rfloor$ functionally independent integrals of the $(1,k)$ sine-Gordon map (\ref{sG1k}) are given by the trace of the Lax matrix $L^{1,k}$:
\begin{equation}\label{Lax1k}
Tr L^{1,k}({\bf x},\lambda) = Tr \left[ \left(
\begin{array}{cc}
qx_0/x_k & \lambda^{-2}/x_k \\
x_0 & q  \end{array} \right)
\prod_{l=0}^{k-1} \left(
\begin{array}{cc}
p & -x_{l+1} \\
-\lambda^2/x_l & p x_{l+1} /x_l \end{array} \right)
     \right],
\end{equation}
where $pq=\alpha$. The individual integrals are given by the coefficients of the various powers of the spectral parameter $\lambda$ in the expansion of the right-hand side of (\ref{Lax1k}).

  It was also shown in \cite{quispel91imd} that the sine-Gordon map $\phi_k$ is either measure preserving or anti measure preserving, i.e. satisfies
\begin{equation}\label{epsmeas}
P({\bf x}') = \epsilon |D\phi_k({\bf x})| P({\bf x}),
\end{equation}
where $P({\bf x}) := \prod_{l=0}^{k} x_l$, and $\epsilon$ is given by (\ref{define_eps}).

  It is easy to see that the rhs of (\ref{Lax1k}) is equal to
\begin{equation}\label{rhsLax1k}
\begin{aligned}
&\mathrm{Tr} \left[ \frac{1}{x_k} \left( \begin{array}{cc}
qx_0 & \lambda^{-2}  \\
x_0x_k & qx_k \end{array} \right)
\prod_{l=0}^{k-1} \frac{1}{x_l} \left( \begin{array}{cc}
px_l & -x_lx_{l+1}  \\
-\lambda^2 & px_{l+1} \end{array} \right) \right]   \\
&= \mathrm{Tr} \left[ \left( \begin{array}{cc} 
qx_0 & \lambda^{-2} \\
x_0x_k & qx_k \end{array} \right) \prod_{l=0}^{k-1} \left(
\begin{array}{cc}
px_l & -x_lx_{l+1}  \\
-\lambda^2 & px_{l+1} \end{array} \right) \right] / \left[ \prod_{l=0}^{k} x_l  \right].
\end{aligned}
\end{equation}

We now recognize that the denominator of the integrals (\ref{rhsLax1k}) equals the Darboux polynomial $P$ in (\ref{epsmeas}). Bearing in mind that the matrices in the trace in (\ref{rhsLax1k}) are all polynomial, it follows using Theorem 1 that this trace is also a Darboux polynomial with the same cofactor $C = \epsilon |D\phi_k({\bf x})|$, for all values of $\lambda$. \hfill $\Box$


\end{document}